\documentclass{amsart}

\usepackage{amsmath}
\usepackage{amsaddr}
\usepackage{systeme}
\usepackage{amsfonts}
\usepackage{framed}
\usepackage[colorlinks=false, urlcolor=black, pdfborder={0 0 0}]{hyperref}
\usepackage{mathrsfs}
\usepackage{amssymb}
\usepackage{blindtext}

\newtheorem{theorem}{Theorem}[section]
\newtheorem{lemma}[theorem]{Lemma}
\newtheorem{proposition}[theorem]{Proposition}
\theoremstyle{definition}
\newtheorem{definition}[theorem]{Definition}
\newtheorem{example}[theorem]{Example}
\newtheorem{corollary}[theorem]{Corollary}

\theoremstyle{remark}
\newtheorem{remark}[theorem]{Remark}

\numberwithin{equation}{section}

\newcommand{\mathbbm}[1]{\text{\usefont{U}{bbm}{m}{n}#1}}
\newcommand{\cP}{ {\mathcal P} }

\newcommand{\bone}{ {\mathbbm{1}} }
\newcommand{\bE}{ {\mathbb{E}} }

\newcommand{\bN}{ {\mathbb{N}} }
\newcommand{\bQ}{ {\mathbb{Q}} }
\newcommand{\bR}{ {\mathbb{R}} }
\newcommand{\bu}{ {\bf u} }

\newcommand{\bZ}{ {\mathbb{Z}} }
\newcommand{\bP}{ {\mathbb{P}} }

\newcommand{\dd}{\overset{d}{=}}
\newcommand{\ddd}{\overset{d}{\Rightarrow}}
\newcommand{\dddd}{\overset{d}{\Leftarrow}}

\newcommand{\ee}{ \varepsilon }

\allowdisplaybreaks[4]
\newcommand{\authorfootnotes}{\renewcommand\thefootnote{\@fnsymbol\c@footnote}}

\makeatletter
\renewcommand{\email}[2][]{%
  \ifx\emails\@empty\relax\else{\g@addto@macro\emails{,\space}}\fi%
  \@ifnotempty{#1}{\g@addto@macro\emails{\textrm{(#1)}\space}}%
  \g@addto@macro\emails{#2}%
}
\makeatother

\title{On discrete-time self-similar processes with stationary increments}
\author{Yi Shen, Zhenyuan Zhang}
\address{Department of Statistics and Actuarial Science, University of Waterloo}

\email{yi.shen@uwaterloo.ca}
\email{z569zhan@edu.uwaterloo.ca}

\begin{document}


\maketitle
%\begin{center}
%
%%  \normalsize
%%  \authorfootnotes
%
%
%
%\end{center}

\begin{abstract}
In this paper we study the self-similar processes with stationary increments in a discrete-time setting. Different from the continuous-time case, it is shown that the scaling function of such a process may not take the form of a power function $b(a)=a^H$. More precisely, its scaling function can belong to one of three types, among which one type is degenerate, one type has a continuous-time counterpart, while the other type is new and unique for the discrete-time setting. We then focus on this last type of processes, construct two classes of examples, and prove a special spectral representation result for the processes of this type. We also derive basic properties of discrete-time self-similar processes with stationary increments of different types.
\end{abstract}

\section{Introduction}
Self-similar processes has been an important research topic in stochastic processes for a long time, due to its technical tractability and various applications in areas such as finance and physics. A general introduction of self-similar processes can be found, for example, in \cite{emb} and \cite{sam16}.

Among self-similar processes, those having stationary increments, abbreviated as ``ss-si processes'', often attract special attention from the researchers. The ss-si processes combine two types of probability symmetries: self-similarity, corresponding to the invariance of the distribution under rescaling, and the stationarity of the increments, corresponding to the invariance of the distribution of the increments under translation. As a result, they possess many desirable properties and include commonly used processes such as fractional Brownian motions and stable L\'{e}vy processes.

The classical setting for self-similar processes is in continuous-time, \textit{i.e.}, $t\in[0,\infty)$. In this case, if a process $\mathbf X=\{X(t)\}_{t\geq 0}$ satisfies that for any $a>0$, there exists $b(a)>0$ such that $\{X(at)\}_{t\geq 0}\dd \{b(a)X(t)\}_{t\geq 0}$, then $\mathbf X$ is said to be self-similar. It is easy to show that if the process is in addition nontrivial and stochastically continuous at 0, then the only possible functions $b$ to make this condition hold are $b(a)=a^H$ for some $H\geq 0$ (\cite{emb}). Consequently, self-similar processes are also often defined as processes $\mathbf X$ such that $\{X(at)\}_{t\geq 0}\dd\{a^HX(t)\}_{t\geq 0}$. It should be noted, however, that the second definition of the self-similar processes is not what the term ``self-similar'' originally or literally means. It is taken as a definition simply because of the equivalence between the two definitions of self-similarity. Logically, if one takes the first definition as the original definition, then the second definition should be regarded as a property of self-similarity.

In this paper we consider dt-ss-si processes, the self-similar processes with stationary increments defined on $\bN_0$, the set of all non-negative integers, instead of on $[0,\infty)$. The self-similarity in discrete-time becomes $\{X_{nm}\}_{m\in\bN_0}\dd\{b(n)X_m\}_{m\in\bN_0}$, where $n$ can only be positive integers now. Interestingly, it turns out that when defined on $\bN_0$, the two definitions of self-similarity are no longer equivalent. More precisely, besides the case where $b(n)=n^H, H>0$ and the degenerate case where $b(n)\equiv 1$, a new possibility $b(n)=(|n|_p)^H$, where $H>0$ and $|n|_p$ is the $p$-adic norm of $n$, arises. As one can see from the later parts of this paper, this change from the continuous-time case is mainly due to the discretization of the possible rescaling factor and the drop of the continuity requirement, which no longer makes sense in the discrete-time setting.

As the new, nondegenerate type in the discrete-time setting, the case where $b(n)=(|n|_p)^H$, $H>0$ is further studied in this paper. Two classes of dt-ss-si processes having such scaling function $b$ are constructed. Moreover, we find that the dt-ss-si processes which are of this type and in $L^2$ have a very particular spectral representation. Very roughly speaking, such a process can always be decomposed into waves with periods of different powers of $p$ and magnitudes decreasing in period.

The rest of this paper is organized as follows: Section 2 introduces basic settings and notations. Section 3 establishes the classification theorem, followed by an embedding result for dt-ss-si processes with $b(n)=n^H$ and basic properties for dt-ss-si processes of different types. Sections 4 and 5 focus on the dt-ss-si processes with $b(n)=(|n|_p)^H$. We give two classes of such processes in Section 4, then state and prove the spectral representation using the notion of almost periodic functions in Section 5.

\section{Basic settings and notations}
Let $\bN_0=\{0,1,\dots\}$ be the set of non-negative integers and $\bN=\{1,2,\dots\}$ be the set of positive integers. We first extend the definition of self-similarity to discrete-time. In order to ensure that the rescaled process is comparable to the original process, the scaling factor must be a positive integer in this case. Therefore, we have
\begin{definition}
A real-valued discrete-time stochastic process $\mathbf X=\{X_m\}_{m\in\bN_0}$ is called a {\it discrete-time self-similar process}, if
for any $n\in\bN$, there exists $b(n)>0$, such that
\begin{equation}
\{X_{nm}\}_{m\in\bN_0}\dd \{b(n)X_m\}_{m\in \bN_0}\label{ss}.
\end{equation}
\end{definition}

Here and later, ``$\dd$'' means equality in the sense of distribution, \textit{i.e.}, the two sides of this symbol have the same distribution.

Denote by $\cP=\{2,3,5,\dots\}$ the set of all primes. It can be easily seen from (\ref{ss}) that the scaling function $b(n)$ for a discrete-time self-similar process must be completely multiplicative, \textit{i.e.}, $b(m)b(n)=b(mn)$ for all $m,n\in\bN$. Consequently, for $n=\prod_{p_i\in\mathcal P}p_i^{r_i}$, $b(n)=\prod_{p_i\in \mathcal P}(b(p_i))^{r_i}$, hence $b(n)$ is determined by its values on $\mathcal P$. On the other hand, any completely multiplicative function $b: \bN\to \mathbb R^+$ is a legitimate scaling function for some discrete-time self-similar process. A simple example is given by $X_0=0, X_n=b(n)X_1$ for $n\geq 1$.

Recall that a discrete-time stochastic process $\mathbf X=\{X_m\}_{m\in\bN_0}$ is said to have {\it stationary increments}, if its increment process is stationary. In other words, for any $k\in\bN$,
\begin{equation}
\{X_{m+1}-X_m\}_{m\in\bN_0}\dd\{X_{m+k+1}-X_{m+k}\}_{m\in\bN_0}.\label{si}
\end{equation}

In this paper, we are mainly interested in the discrete-time self-similar processes with stationary increments, {\it dt-ss-si} processes. They are the processes that satisfy both (\ref{ss}) and (\ref{si}).

\section{Classification and properties of dt-ss-csi processes}
In this part we show that the dt-ss-si processes can be classified into three types according to their scaling properties. Among these three types, one has a continuous-time counterpart, one is degenerate, while the other only exists for the discrete case. It turns out that the results in this section actually work for a larger family of processes for which both the self-similarity and the stationarity of the increments only hold marginally. Moreover, the stationarity of the increments can be relaxed to the cyclostationarity with any fixed integer period. We begin this section by generalizing these notions, and will work with the processes which are marginally self-similar with marginally cyclostationary increments in this section.

\begin{definition}
Let $\mathbf X=\{X_m\}_{m\in\bN_0}$ be a discrete-time stochastic process. If (\ref{ss}) holds marginally, \textit{i.e.}, for any $m\in\bN_0$ and $n\in \bN$,
\begin{equation}
X_{nm}\dd b(n)X_m\label{e:marginalss},
\end{equation}
then $\mathbf X$ is called \textit{marginally self-similar}.
\end{definition}

\begin{definition}
Let $\tau\in\bN$. A stochastic process $\mathbf X=\{X_m\}_{m\in\bN_0}$ is said to have {\it marginally cyclostationary increments with period $\tau$}, if for any $k,m\in\bN_0$,
\begin{equation}
X_{m+1}-X_m\dd X_{m+k\tau+1}-X_{m+k\tau}.\label{csi}
\end{equation}
\end{definition}

A discrete-time, marginally self-similar process having marginally cyclostationary increments with period $\tau$ is denoted as \textit{dt-ss-csi($\tau$)}, or \textit{dt-ss-csi} when it is not necessary to specify the value of $\tau$.

The case $\mathbf X\equiv 0$ being trivial, we assume the distribution of $X_1$ is nondegenerate. For any probability distribution $F$ on $(\mathbb R, \mathcal B(\mathbb R))$ and $a\in\mathbb R$, denote by ``$Y/a\sim F$'' the relation that $Y\in B=F(aB)$ for any Borel set $B\subset\mathbb R$, where $aB=\{ax: x\in B\}$. Here and later, we always identify a probability distribution on $(\mathbb R, \mathcal B(\mathbb R))$ with its distribution function.\\

The main result of this section is the following Ostrowski-type classification theorem. Note that since the marginal distributions of $\mathbf X$ are not necessarily in $L^1$, we do not have the triangle inequality required by a direct application of the classical Ostrowski's theorem.

\begin{theorem}\label{t:classification}
The scaling function $b(n)$ of a dt-ss-csi($\tau$) process must be one of the followings:
\begin{enumerate}
\item $b(n)=1$ for all $n\in\bN$.
\item There exist a unique prime $p$ and $H>0$ such that $b(n)=(|n|_p)^H$. In other words, $b(p)<1$ and $b(q)=1$ for $q\in\mathcal P, q\neq p$.
\item There exists $H>0$ such that $b(n)=n^H$ for all $n\in\bN$.
\end{enumerate}
Conversely, for any completely multiplicative function $b(n)$ on $\bN$ satisfying one of the conditions above, there exists a non-trivial dt-ss-si process having $b(n)$ as its scaling function.
\end{theorem}

Some preparatory results are needed to prove Theorem \ref{t:classification}.

\begin{proposition}\label{p:general}
Let $G_1, G_2, G_3$ be probability distributions on $\mathbb R$. Assume $G_1$ is not concentrated at 0. Then there exist constants $k_2,k_3$, such that for
$a_1, a_2, a_3\geq 0$,
\begin{equation}\label{e:feasiblecondition}
a_1>k_2 a_2+k_3 a_3
\end{equation}
implies that there do not exist random variables $Y_i$, $i=1,2,3$, satisfying $Y_i/a_i\sim G_i$ and $\sum_{i=1}^3Y_i=0$.
\end{proposition}

\begin{proof}
Let $Q_i(t)$ be the quantile function of $G_i$:
$$
Q_i(t)=\inf\{x\in\mathbb R: G_i(x)\geq t\}, \quad t\in(0,1), i=1,2,3.
$$
Note that $Q_i$ is non-decreasing. Moreover, as $G_1$ is not concentrated at 0, either $G_1((0,\infty))>0$, or $G_1((-\infty, 0))>0$.

First assume $G_1((0,\infty))>0$. Then there exists $s\in(0,1)$, such that $Q_1(s)>0$. For $0<c<d<1$, define the average quantile functional from $c$ to $d$:
$$
\bar{Q}_i([c,d])=\frac{1}{d-c}\int_c^d Q_i(t)dt.
$$

We use a result in \cite{puc}, where $\sum_{i=1}^3Y_i=0$ is translated into the ``joint mixability'' of the distributions of $Y_1, Y_2, Y_3$ with center 0. By linearity of the average quantile functional, the average quantile functionals for $Y_i$ satisfying $Y_i/a_i\sim G_i$, denoted by $\bar{q}_{i}$, satisfy
$$
\bar{q}_{i}([c,d])=a_i\bar{Q}_i([c,d]), \quad 0<c<d<1.
$$

Take $\beta_1=s$, $\beta_2, \beta_3>0$ such that $\beta=\sum_{i=1}^3\beta_i<1$. By Proposition 3.3 in \cite{puc}, a necessary condition for the distributions of $Y_i, i=1,2,3$ to be jointly mixable is
\begin{equation}\label{e:generala}
\sum_{i=1}^3\bar{q}_{i}([\beta_i, \beta_i+1-\beta])\leq 0.
\end{equation}
(\ref{e:generala}) implies that
\begin{equation}\label{e:detailedcond}
a_1\bar{Q}_1([s, s+1-\beta])\leq -\sum_{i=2,3}a_i\bar{Q}_i([\beta_i, \beta_i+1-\beta]).
\end{equation}
Note that $\bar{Q}_1([s, s+1-\beta])>0$ by construction. Thus, it suffices to take
$$
k_i=-\frac{\bar{Q}_i([\beta_i, \beta_i+1-\beta])}{\bar{Q}_1([s, s+1-\beta])}, \quad i=2,3.
$$

For the other case, assume that $G_1((-\infty, 0))>0$. As a result, there exists $s\in(0,1)$, such that $Q_1(s)<0$. Similarly as in the previous case, Proposition 3.3 in \cite{puc} gives another necessary condition for the distributions of $Y_i, i=1,2,3$ to be jointly mixable:
\begin{equation}\label{e:generalb}
\sum_{i=1}^3\bar{q}_{i}([\beta-\beta_i, 1-\beta_i])\geq 0.
\end{equation}
By taking $\beta_1=1-s$, $\beta_2, \beta_3>0$ such that $\beta=\sum_{i=1}^3\beta_i<1$, (\ref{e:generalb}) becomes
$$
a_1\bar{Q}_1([s-(1-\beta), s])\geq -\sum_{i=2,3}a_i\bar{Q}_i([\beta-\beta_i, 1-\beta_i]).
$$
Recall that since $Q_1(s)<0$ and $Q_1$ is non-decreasing, $\bar{Q}_1([s-(1-\beta), s])<0$. Hence it suffices to take
$$
k_i=-\frac{\bar{Q}_i([\beta-\beta_i, 1-\beta_i])}{\bar{Q}_1([s-(1-\beta), s])}, \quad i=2,3.
$$
\end{proof}

\begin{proposition}\label{p:samedist}
Under the same setting as in Proposition \ref{p:general}, if in addition, $G_1$ and $G_2$ satisfy $G_1(B)=G_2(-B)$ for any Borel set $B\subset\mathbb R$, then for every $k_2>1$, there exists $k_3\in\mathbb R$ such that the result in Proposition \ref{p:general} holds.
\end{proposition}

\begin{proof}
We prove the case where $G_1((0,\infty))>0$. The case where $G_1((-\infty,0))>0$ is symmetric.

In the proof of Proposition \ref{p:general}, since $Q_1$ is non-decreasing and not constantly 0 on $(0,1)$, there exists $s\in\mathbb R$ such that $Q_1(s)>0$, and $Q_1$ is continuous at $s$. As a result, there exists $\epsilon\in (0, s\wedge (1-s))$ satisfying
$$
\bar{Q}_1([s-\epsilon, s])>\frac{1}{k_2}\bar{Q}_1([s,s+\epsilon])>0.
$$
Moreover, note that $\bar{Q}_1([c,d])=-\bar{Q}_2([1-d, 1-c])$ for all $0<c<d<1$. Taking $\beta_1=s-\epsilon$, $\beta_2=1-s-\epsilon$ and $\beta_3=\epsilon$, (\ref{e:generala}) becomes
$$
a_1\bar{Q}_1([s-\epsilon,s])\leq a_2\bar{Q}_1([s, s+\epsilon])-a_3\bar{Q}_3([\epsilon, 2\epsilon]).
$$
It suffices to take $k_3=-\frac{\bar{Q}_3([\epsilon, 2\epsilon])}{\bar{Q}_1([s-\epsilon,s])}$.
\end{proof}

As immediate consequences of Propositions \ref{p:general} and \ref{p:samedist}, we have

\begin{corollary}\label{c:feasible}
Let $\{X_n\}_{n\in \bN_0}$ be a dt-ss-csi($\tau$) process, and $b(n), n\in\bN$ be its scaling function. Then for any $m\in\bN, k>1$, there exists $c_m\in \mathbb R$, such that
$$
b(n+m)\leq k b(n)+c_m, \quad n\in\bN_0.
$$
\end{corollary}

\begin{proof}
By the cyclostationarity of the increments, $X_{n+m}-X_{n}\dd X_{n'+m}-X_{n'}$, where $n'$ is the residue of $n$ modulo $\tau$. Then by Proposition \ref{p:samedist} with $G_1, G_2, G_3$ being the distributions of $X_1, -X_1, X_{n'}-X_{n'+m}$, $a_1=b(n+m), a_2=b(n), a_3=1$, there exists $c_{n',m}\in \mathbb R$ such that $b(n+m)\leq k b(n)+c_{n',m}$. It remains to take $c_m=\bigvee_{n'=0}^{\tau-1}c_{n',m}$.
\end{proof}

\begin{corollary}\label{c:smallfeasible}
Let $\{X_n\}_{n\in \bN_0}$ be a dt-ss-csi($\tau$) process, and $b(n), n\in\bN$ be its scaling function, which is not identically 1. Then for any $m\in\bN$, there exists $d_m>0$, such that $b(n\tau)\vee b(n\tau+m)\geq d_m$ for all $n\in\bN_0$.
\end{corollary}

\begin{proof}
Similarly as in the proof of Corollary \ref{c:feasible}, $X_{n\tau+m}-X_{n\tau}\dd X_{m}-X_{0}=X_m$, where the second equality holds since the (marginal) self-similarity clearly implies $X_0=0$ almost surely when $b(n)\not\equiv 1$. Applying Proposition \ref{p:general} with $G_1$, $G_2$ and $G_3$ being the distributions of $X_{m}, X_1, -X_1$ respectively, and $a_1=1, a_2=b(n\tau), a_3=b(n\tau+m)$, there exist constants $k_{2,m}$ and $k_{3,m}$, such that $1\leq k_{2,m}b(n\tau)+k_{3,m}b(n\tau+m)$. Moreover, as $b(n\tau)$ and $b(n\tau+m)$ are non-negative, $k_{2,m}$ and $k_{3,m}$ can be chosen to be strictly positive, hence $b(n\tau)\vee b(n\tau+m)\geq \frac{1}{k_{2,m}+k_{3,m}}=:d_m$.
\end{proof}

\begin{proof}[Proof of Theorem \ref{t:classification}]
Define function $f(p):=\log_p b(p)$. Let $A=\{f(p): p\in\cP\}$ be the set of possible values of $f$ for prime numbers. We first prove that $A$ is bounded from above by contradiction. Suppose $\sup(A)=\infty$. Then for any $L>0$, $\cP^{L}:=\{p\in\cP:f(p)\geq L\}$ is not empty. Choose $L$ large enough such that $2\not\in\cP^{L}$. Denote by $p_L$ the smallest element in $\cP^L$. Then for any $n<p_L$, $b(n)<n^L$. Hence
\begin{equation}\label{e:brelation}
b(p_L-1)=b(2)b\left(\frac{p_L-1}{2}\right)<b(2)\left(\frac{p_L-1}{2}\right)^L<\frac{b(2)}{2^L}(p_L)^L\leq \frac{b(2)}{2^L}b(p_L).
\end{equation}
Since $b(2)$ is a fixed constant, $b(p_L)/b(p_L-1)\to\infty$ as $L\to\infty$. Thus, for any $k>1$ and $c\in\mathbb R$, there exists $L$ large enough such that $b(p_L)>kb(p_L-1)+c$, contradicting Corollary \ref{c:feasible}. Hence $A$ must be bounded from above.

Next we show that if $\sup(A)>0$, then this supremum must be achieved by some $p\in\mathcal P$. Suppose there doesn't exist $p\in\cP$ such that $f(p)=H:=\sup (A)>0$. In particular, $\ee:=H-f(2)>0$. For each $n\in\bN$, let $p_n$ be the smallest prime such that $f(p_n)>H-\frac{1}{n}$. Thus the sequence $\{p_n\}$ and $\{f(p_n)\}$ are both non-decreasing, with limits $\infty$ and $H$ respectively. Let $N\in\bN$ be such that $1/N<\ee$ and $H-1/N>0$, then for $n\geq N$, we have by a similar argument as (\ref{e:brelation}),
\begin{align}
\frac{b(p_n)}{b(p_n-1)}\geq \frac{2^{H-1/n}}{b(2)}\geq 2^{\ee-1/N}=:K>1.\label{2K}
\end{align}
This contradicts Corollary \ref{c:feasible}, as for $k\in(1,K)$ and any $c_1\in \mathbb R$, since $b(p_n)\geq p_n^{H-1/N}\to\infty$ as $n\to\infty$, $b(p_n)>kb(p_n-1)+c$ for $n$ large enough.

As a result, if $\sup(A)>0$, there must exist $p\in\mathcal P$, such that $f(p)=\sup (A)$. We show that in this case, $f(q)=\sup (A)$ for any $q\in \mathcal P$. As a result, $b(n)=n^H$ for $n\in\bN$ and $H=\sup (A)>0$. Suppose this is not true, then there exists $q\in \mathcal P$ satisfying $f(q)<f(p)$. For each $r\in \bN$ satisfying $p^r>q$, there exists $m\in\{1,\cdots, q-1\}$, such that $q|p^r-m$. By Corollary \ref{c:feasible}, for any $k>1$, we have
\begin{equation}\label{e:prandprm}
b(p^r)\leq kb(p^r-m)+c_m\leq kb(p^r-m)+c,
\end{equation}
where $c=\vee_{m=1}^{q-1}c_m$. However, note that
$$
b(p^r-m)=b(q)b\left(\frac{p^r-m}{q}\right)\leq b(q)\left(\frac{p^r-m}{q}\right)^H<\frac{b(q)}{q^H}p^{rH}=\frac{b(q)}{q^H}b(p^r).
$$
By the choice of $q$, $\frac{b(q)}{q^H}<1$. Hence (\ref{e:prandprm}) can not hold for $k\in(1,\frac{q^H}{b(q)})$ and $r$ large enough. Thus, we conclude that $f(q)=\sup (A)$, and consequently, $b(n)=n^H$.

It remains to consider the case where $\sup(A)\leq 0$, which is equivalent to $b(n)\leq 1$ for all $n\in\bN$. Suppose there exist two distinct primes $p, q\in\mathcal P$ such that $b(p)<1$ and $b(q)<1$. Let $d_m, m=1, \dots, \tau$ be as given in Corollary \ref{c:smallfeasible} and define $d=\bigwedge_{m=1}^\tau d_m$. Take $i,j$ large enough so that $(b(p))^i<d$ and $(b(q))^j<d$. By B\'{e}zout's lemma, there exist $M,N\in \bN$ such that $0<Mp^i-Nq^j\tau\leq \tau$. Corollary \ref{c:smallfeasible} then implies that $b(Mp^i)\vee b(Nq^j\tau)\geq d$. However, by the choices of $i$ and $j$ we have $b(Mp^i)\leq b(p^i)<d$ and $b(Nq^j\tau)\leq b(q^j)<d$, contradiction. Therefore there exists at most one prime $p$ such that $b(p)<1$. This leads to cases (1) and (2).

Finally, for any completely multiplicative function $b(n)$ satisfying one of the three conditions in Theorem \ref{t:classification}, there exists a non-trivial dt-ss-si process having $b(n)$ as its scaling function, according to Examples \ref{ex:trivial}, \ref{e:iid}, \ref{e:periodic} and Theorem \ref{t:embedding} that we will see.
\end{proof}

 It should be pointed out that a similar result was obtained in \cite{gef} for second-order dt-ss-si processes, \textit{i.e.}, the processes in $L^2$ whose covariance functions satisfy properties related to the self-similarity and the stationarity of the increments of the process. In this sense, Theorem \ref{t:classification} can be regarded as a generalization of that result to the general dt-ss-si processes which are not necessarily in $L^2$.

\begin{example}\label{ex:trivial}
Let $X_n, n\in \bN_0$ be independent and identically distributed random variables, then $\{X_n\}_{n\in\bN_0}$ is a trivial example of a dt-ss-si process with $b(n)=1, n\in\bN$.
\end{example}

We call the dt-ss-si processes with scaling functions satisfying the three cases in Theorem \ref{t:classification} dt-ss-si processes of types I, II, III, respectively. Type III is what people are familiar with from the continuous-time ss-si processes. The following theorem shows that there is indeed a correspondence between the continuous-time ss-si processes and the dt-ss-si processes of type III.

\begin{theorem}\label{t:embedding}
If $\{X(t)\}_{t\geq 0}$ is an ss-si process, then $\{X_n\}_{n\in\bN_0}$ given by $X_n=X(n), n\in\bN_0$ is a dt-ss-si process. Conversely, if $\{X_n\}_{n\in\bN_0}$ is a dt-ss-si process with scaling function $b(n)=n^H$ for $H>0$, then there exists a unique in distribution ss-si process $\{X(t)\}_{t\geq 0}$, such that $\{X(n)\}_{n\in\bN_0}\dd \{X_n\}_{n\in \bN_0}$.
\end{theorem}

\begin{proof}
An ss-si process observed at discrete-time $\bN_0$ is clearly a dt-ss-si process, hence we focus on the other direction. For that purpose, we will derive the distribution of the ss-si process, $\mathbf Y=\{Y(t)\}_{t\geq 0}$, from any arbitrary dt-ss-si process $\mathbf X=\{X_n\}_{n\in\bN_0}$, so that they have the same distribution on $\bN_0$.

First, it is not difficult to determine the distribution of $\mathbf Y$ on $\mathbb Q^+=\mathbb Q\cap [0,\infty)$ by self-similarity:
\begin{equation}\label{e:rationaldist}
(Y(s_1/t_1),\dots,Y(s_n/t_n))\dd(t_1t_2\dots t_n)^{-H}(X_{s_1t_2\dots t_n},\dots,X_{t_1\dots t_{n-1}s_n})
\end{equation}
for $\{s_i\}_{i=1}^n\subset\bN_0, \{t_i\}_{i=1}^n\subset\bN$. This distribution does not depend on the choice of $s_i$ and $t_i, i=1,\dots, n$. Moreover, since the original finite-dimensional distributions on $\bN_0$ are consistent, the finite-dimensional distributions on $\mathbb Q^+$ are also consistent. Hence, by Kolmogorov's extension theorem, such a process $\{Y(r)\}_{r\in \mathbb Q^+}$ exists. One can check that $\{Y(r)\}_{r\in\mathbb Q^+}$ is ss-si on $\bQ^+$. Indeed, for any $n\in\bN,\ p,s_1, \dots, s_n\in\bN_0$ and $\ q,t_1, \dots t_n\in\bN$,
\begin{align}\label{qss}
\begin{split}
\{Y(s_i/t_i)\}_{i=1, \dots, n}&\dd(\prod_{i=1}^n t_i)^{-H}\{X_{(s_i/t_i)\prod_{i=1}^nt_i}\}_{i=1,\dots, n}\\
&\dd(pq^{n-1}\prod_{i=1}^n t_i)^{-H}\{X_{pq^{n-1}(s_i/t_i)\prod_{i=1}^nt_i}\}_{i=1,\dots, n}\\
&\dd (p/q)^{-H}\{Y(s_ip/t_iq)\}_{i=1,\dots, n}.
\end{split}
\end{align}
Also, by the stationarity of the increments of $\{X_n\}_{n\in\bN_0}$,
\begin{align}\label{qsi}
\begin{split}
&~\{Y(s_i/t_i)\}_{i=1,\dots, n}\\
\dd&~(q^n\prod_{i=1}^n t_i)^{-H}\{X_{q^n(s_i/t_i)\prod_{i=1}^nt_i}\}_{i=1, \dots, n}\\
\dd&~(q^n\prod_{i=1}^n t_i)^{-H}\{X_{(s_iq^nt_i^{-1}+pq^{n-1})\prod_{i=1}^nt_i}-X_{pq^{n-1}\prod_{i=1}^nt_i}\}_{i=1, \dots, n}\\
\dd&~(q^n\prod_{i=1}^n t_i)^{-H}(X_{(s_iq+pt_i)q^{n-1}t_i^{-1}\prod_{i=1}^nt_i}-X_{pq^{n-1}\prod_{i=1}^nt_i}\}_{i=1, \dots, n}\\
\dd&~\{Y(s_i/t_i+p/q)-Y(p/q)\}_{i=1, \dots, n}.
\end{split}
\end{align}

Finally, as it is proved in \cite{ver} that every ss-si process with $H>0$ is stochastically continuous, the distribution on $\mathbb Q^+$ uniquely extends to the distribution on $[0,\infty)$. The self-similarity and the stationarity of the increments are naturally inherited. Thus, we conclude that any dt-ss-si process with $b(n)=n^H, H>0$ determines a unique in distribution ss-si process, which has the same distribution on $\bN_0$ as the dt-ss-si process.
\end{proof}

The following proposition collects several basic properties for dt-ss-si processes of type III. They are direct consequences of Theorem \ref{t:embedding} and the corresponding results in continuous-time, which we cite individually.

\begin{proposition}
Let $\{X_n\}_{n\in\bN_0}$ be a dt-ss-si process of type III with $b(n)=n^H, H>0$, then
\begin{enumerate}
\item \cite{emb} $X_0=0$ almost surely.
\item \cite{obr} For $m,n\in\bN_0$, $m\neq n$, $\bP(X_m=X_n)=\bP(X_m=X_n=0)=\bP(X_i=0, i\in\bN_0)$.
\item \cite{obr} If $H\neq 1$, then $X_1$ has no atom except possibly at zero.
\item \cite{emb} If $0<H<1$ and $\bE(|X_1|)<\infty$, then $\bE(X_n)=0$ for all $n\in\bN_0$.
\item \cite{sam} If $\bE(|X_1|^\gamma)<\infty,$ $0<H<{1/\gamma}$ for $0<\gamma<1$ and $0<H\leq 1$ for $\gamma\geq 1$.
\item \cite{sam} If $\bE(X_1^2)<\infty$, then
\[
\text{Cov}(X_n,X_m)=\frac{1}{2}(n^{2H}+m^{2H}-|n-m|^{2H})\text{Var}(X_1).
\]
\end{enumerate}
\end{proposition}

More interestingly, for the dt-ss-si processes of type II, which do not find their counterparts in continuous-time, we have

\begin{proposition}\label{p:type2property}
Let $\{X_n\}_{n\in\bN_0}$ be a dt-ss-si process of type II with $b(n)=(|n|_p)^H$ for some $H>0$, then:
\begin{enumerate}
\item $X_0=0$ almost surely.
\item $\{X_n\}_{n\in\bN_0}$ is recurrent, in the sense that each $X_n$ is a limit point of $\{X_n\}_{n\in\bN_0}$ almost surely.
\item $X_1\dd-X_1$. Consequently, $\bE(|X_1|)<\infty$ implies $\bE(X_n)=0$ for all $n\in\bN_0$.
\item For $m,n\in \bN_0$, $m\neq n$, $\bP(X_m=X_n)=\bP(X_m=X_n=0)=\bP(X_i=0, i\in\bN_0)$.
\item $X_1$ has no atom except possibly at zero.
\item Let $a=\sup\{x\geq 0:\ \bP(|X_1|<x)=0\}$ and $b=\inf\{x>0:\ \bP(|X_1|>x)=0\}$, then $b\geq (1+2p^{-H})a.$
\item If $\bE(X_1^2)<\infty$, then for any $m,n\in\bN_0$,
$$
\text{Cov}(X_n,X_m)=\frac{1}{2}\left((|n|_p)^{2H}+(|m|_p)^{2H}-(|n-m|_p)^{2H}\right)\text{Var}(X_1).
$$
\end{enumerate}
\end{proposition}

\begin{proof}
(1) and (2) are trivial from definition. For (3), note that by the stationarity of the increments, for any $n\in\bN$,
\[
X_{p^n+1}-X_{p^n}\dd X_1\dd X_{p^n}-X_{p^n-1}.
\]
Since $X_{p^n}\dd p^{-nH}X_1, H>0$, $X_{p^n}\to 0$ in distribution and hence in probability as $n\to\infty$. We thus have
\[
X_1\dddd X_{p^n+1}\dd X_{p^n-1}\ddd -X_1.
\]

(4) We have for $m,n\in\bN_0$ and any $M>0$,
\begin{align*}
\bP(X_m=X_n\neq 0)&\leq \bP(0<|X_m|< M)+\bP(|X_n|\geq M)\\
&=\bP(0<|X_m|<M)+\bP(|X_1|\geq Mb(n)^{-1})
\end{align*}
thus equation (14) in \cite{obr} can be replaced by
\[
\lim_{n\to\infty}\bP(X_m=X_{p^n}\neq 0)=0.
\]
The rest of the proof follows in the same way as in the proof of Lemma 3 of \cite{obr}.\\

(5) Suppose $\bP(X_1=x)=p>0$ for some $x\neq 0$. Choose $\ee>0$ such that $\bP(X_1\in (x-\ee,x+\ee)\setminus \{x\})<p/2$. Choose $n$ large enough such that $\bP(|X_{p^n}|\geq\ee)<p/2$, then
\begin{align*}
p=&~\bP(X_1=x)\\
=&~\bP(X_1=x,\ |X_{p^n+1}-X_1|\geq \ee)+\bP(X_1=x,\ |X_{p^n+1}-X_1|=0)\\
&+\bP(X_1=x,\ 0<|X_{p^n+1}-X_1|< \ee)\\
\leq&~ \bP(|X_{p^n+1}-X_1|\geq \ee)+\bP(X_{p^n+1}=X_1\neq 0)\\
&+\bP(X_1=x,\ X_{p^n+1}\in (x-\ee,x+\ee)\setminus \{x\}).
\end{align*}
The second term in the last expression is 0 by property (4), hence
\begin{align*}
p&\leq \bP(|X_{p^n+1}-X_1|\geq \ee)+\bP(X_{p^n+1}\in (x-\ee,x+\ee)\setminus \{x\})\\
&= \bP(|X_{p^n}|\geq\ee)+\bP(X_1\in (x-\ee,x+\ee)\setminus \{x\})\\
&<p/2+p/2=p,
\end{align*}
which gives a contradiction. Hence $X_1$ does not have any atom except for 0.\\

(6) Let $c:=p^{-H}$. Suppose $b<(1+2c)a$, then $\log_c(\frac{b-a}{b})-\log_c(\frac{2a}{b})>1$. So there exists $n\in\bN$ satisfying $b-a<c^nb<2a$. Note that there exist $X,Y,Z$ such that $Y/c^n\dd X\dd Z\dd X_1$, $Y=Z+X$. We have
\[
\bP(X>0, Z>0)=\bP(X>a,Z>a)\leq \bP(Y>2a)\leq\bP(c^{-n}Y>b)=0.
\]
Symmetrically, $\bP(X<0, Z<0)=0$.
Meanwhile,
\[
\bP(|X|>c^{-n}(b-a))=\bP(|Y|>b-a)=\bP(|X+Z|>b-a).
\]
However, by the definition of $a$ and $b$, $|X+Z|>b-a$ implies that almost surely, $X>0, Z>0$ or $X<0, Z<0$. As a result, $\bP(|X|>c^{-n}(b-a))=0$, contradicting the choice of $b$ since $c^{-n}(b-a)<b$.\\

(7) is trivial by polarization.
\end{proof}

\section{Examples of dt-ss-si processes of type II}
As shown in the previous section, the dt-ss-si processes can be classified into three types. Type I is degenerate and type III has continuous-time counterparts. Type II, for which $b(n)=(|n|_p)^H$, only exists in the discrete-time setting and is, therefore, of special interest. Sections 4 and 5 are mainly dedicated to the study of this type. In this section, we give two classes of examples for dt-ss-si processes of type II.

\begin{example}\label{e:iid}
Let $p\in\cP$ and $0<b<1$. Let $\{Y_n^k\}_{k\in\bN_0,\ 0\leq n\leq p^{k+1}-1}$ be a sequence of independent and identically distributed random variables having any non-degenerate distribution such that $\bE(|Y_0^0-Y_0^1|^q)<\infty$ for some $q>0$. Sufficient conditions for this can be $Y_0^0\in L^1(\Omega, \mathcal F, \mathbb P)$, or $Y_0^0$ is $\alpha$-stable with $0<\alpha\leq 2$. Extend the sequence periodically to $\{Y_n^k\}_{k,n\in\bN_0}$ by defining $Y_\ell^k=Y_n^k$ for $\ell\equiv n$ (mod $p^{k+1}$). Define
\begin{align}
X_n=\sum_{k=0}^\infty b^k(Y_n^k-Y_0^k), \quad n\in\bN_0.    \label{xn11}
\end{align}
It is easy to see that the above summation converges almost surely for any $n\in\bN_0$, thus $\{X_n\}_{n\in\bN_0}$ is well-defined. Indeed,
\[
\sum_{k=0}^\infty \bE(|b^k(Y_n^k-Y_0^k)|^q)=\bE(|Y_0^0-Y_0^1|^q)\sum_{k=0}^\infty b^{qk}<\infty.
\]
$\{X_n\}_{n\in\bN_0}$ is a dt-ss-si process of type II. We show this in the following proposition.
\end{example}

\begin{proposition}
The process $\{X_n\}_{n\in\bN_0}$ given in (\ref{xn11}) is dt-ss-si with scaling function $b(n)=(|n|_p)^H$, where $H=-\log_p(b)$.
\end{proposition}

\begin{proof}
We first show that $\{X_{qn}\}_{n\in\bN_0}\dd \{X_n\}_{n\in\bN_0}$ for $q\in\mathcal P, q\neq p$. Note that for any fixed $k\in \bN_0$, by the periodicity of $Y_n^k$, $\{Y^k_{qn}\}_{0\leq n\leq p^{k+1}-1}$ is just a permutation of $\{Y^k_{n}\}_{0\leq n\leq p^{k+1}-1}$, hence also a sequence of independent and identically distributed random variables. Moreover, both $\{Y^k_n\}_{n\in\bN_0}$ and $\{Y^k_{qn}\}_{n\in\bN_0}$ have period $p^{k+1}$ with respect to $n$. Thus,
\[
\{Y^k_{qn}\}_{n\in\bN_0}\dd\{Y^k_{n}\}_{n\in\bN_0},
\]
which clearly implies
\[
\{Y^k_{qn}-Y_0^k\}_{n\in\bN_0}\dd\{Y^k_{n}-Y_0^k\}_{n\in\bN_0}.
\]
Since the sequences with different values of $k$ are independent,
\[
\left\{\sum_{k=0}^\infty b^k(Y^k_{qn}-Y_0^k)\right\}_{n\in\bN_0}\dd\left\{\sum_{k=0}^\infty b^k(Y^k_{n}-Y_0^k)\right\}_{n\in\bN_0},
\]
\textit{i.e.}, $\{X_{qn}\}_{n\in\bN_0}\dd \{X_n\}_{n\in\bN_0}$.

To show $\{X_{pn}\}_{n\in\bN_0}\dd \{bX_n\}_{n\in\bN_0}$, note that by independence,
\[
\{Y_n^k\}_{0\leq n\leq p^{k+1}-1}\dd \{Y_{pn}^{k+1}\}_{0\leq n\leq p^{k+1}-1}.
\]
Since both $\{Y_n^k\}_{n\in\bN_0}$ and $\{Y_{pn}^{k+1}\}_{n\in\bN_0}$ have the same period $p^{k+1}$,
\[
\{Y_n^k\}_{n\in\bN_0}\dd \{Y_{pn}^{k+1}\}_{n\in\bN_0},
\]
hence
\[
\{Y_n^k-Y_0^k\}_{n\in\bN_0}\dd \{Y_{pn}^{k+1}-Y_0^{k+1}\}_{n\in\bN_0}.
\]
As the components with different values of $k$ are independent, we have
\begin{align*}
\left\{\sum_{k=0}^\infty b^k(Y^k_{n}-Y_0^k)\right\}_{n\in\bN_0}&\dd\left\{\sum_{k=0}^\infty b^k(Y^{k+1}_{np}-Y_0^{k+1})\right\}_{n\in\bN_0}\\
&= \left\{b^{-1}\sum_{k=0}^\infty b^{k+1}(Y^{k+1}_{np}-Y_0^{k+1})\right\}_{n\in\bN_0}\\
&= \left\{b^{-1}\sum_{k=0}^\infty b^k(Y^{k}_{np}-Y_0^{k})\right\}_{n\in\bN_0},
\end{align*}
where the last equality follows from $Y^0_{np}=Y^0_{0}, n\in\bN_0$.
Therefore, according to (\ref{xn11}), we have $\{X_{pn}\}_{n\in\bN_0}\dd \{bX_n\}_{n\in\bN_0}$.

Finally, to show the stationarity of the increments, note that the process $\{Y^k_n\}_{n\in\bN_0}$ is stationary, so $\{Y^k_n-Y^k_0\}_{n\in\bN_0}$ has stationary increments for all $k\in\bN_0$. Again by the independence of the components with different values of $k$, $\{X_n\}_{n\in\bN_0}$ has stationary increments.
\end{proof}

\begin{remark}
When $Y_0^0$ follows a Gaussian distribution, $\{X_n\}_{n\in\bN_0}$ is a Gaussian process with covariance function specified in Proposition \ref{p:type2property}, property (7).
\end{remark}

\begin{remark}
Example \ref{e:iid} answers several questions for which the continuous-time counterparts are still open. For instance, as mentioned in \cite{obr} and \cite{ver}, it is not clear whether for a continuous-time ss-si processes with $0<H<1$, denoted by $\{X(t)\}_{t\geq 0}$, the support of $X(1)$ must be unbounded. By Example \ref{e:iid}, we know this is not true for dt-ss-si processes of type II when the support of $Y_0^0$ is bounded.

Another open problem raised in \cite{obr} asks whether the distribution of $X(1)$ must be absolutely continuous on $\mathbb R\setminus \{0\}$ when $H>0$ and $H\neq 1$. The answer is also negative for our dt-ss-si process of type II. It is easy to see that $X_n$ can be expressed in the form
\[
X_n=\sum_{k=0}^\infty b^kX_n^{(k)}
\]
where $\{X_n^{(k)}\}_{k\in\bN_0}$ is a sequence of independent and identically distributed random variables. When the support of $Y_0^0$ is finite, this corresponds to a generalization of the Bernoulli convolution in \cite{per}. When $b$ is a reciprocal of a Pisot number in a certain interval, the distribution of $X_n$ will be singular. This is also the case when $b$ is close enough to $0$, where the support of $X_n$ is a Cantor-type set, again provided that the support of $Y^0_0$ is finite.
\end{remark}

\begin{example}\label{e:periodic}
Fix $p\in\cP,\ b\in(0,1)$. Let $\bu$ be a $p$-dimensional random vector whose entries $u_0,\dots,u_{p-1}$ sum up to 0. For $k\in\bN_0$, let $\{V_k(n)\}_{n\in\bN_0}$ be the stochastic process given by

\begin{align*}
V_k(n):=\begin{cases}
0 & \quad \text{if }p^k\nmid n\\
u_s & \quad \text{if }n\equiv sp^k\ (\text{mod }p^{k+1}).
\end{cases}%\label{fkn}
\end{align*}
Let $\{U^j_k\}_{k\in\bN_0,1\leq j\leq p^{k+1}-1}$ be a sequence of independent random variables such that $U^j_k$ is uniformly distributed on $\{0,1,\dots,p^{k+1}-1\}$, and $\{U^j_k\}_{k\in\bN_0,1\leq j\leq p^{k+1}-1}$ is independent of $\bu$. For each $k\in\bN_0,\ 1\leq j\leq p^{k+1}-1$, define
$$
Y_k^j(n):=V_k(n+U^j_k), \quad n\in\bN_0.\label{ykn}
$$
It is easy to see that $\{Y_k^j(n)\}_{n\in\bN_0}$ is stationary, has period $p^{k+1}$, and the sum in each period is zero since the sum of the entries of $\bu$ is zero. Moreover, these stationary sequences are independent conditional on $\bu$ by the independence of $\{U_k^j\}_{k\in\bN_0, 1\leq j\leq p^{k+1}-1}$. For $j=1, 2, \dots, p^{k+1}-1$, define
$$
Y_{k,j}(n)=\sum_{m=j(n-1)+1}^{jn}Y_k^j(m), \quad n\in\bN,\label{ykjn}
$$
which has period $p^{k+1}$ and the sum in each period is again zero. Let $\{J_k\}_{k\in\bN_0}$ be another sequence of independent uniform random variables on $\{0,1,\dots,p^{k+1}-1\}$, independent of $\bu$ and $\{U_k^j\}_{k\in\bN_0,1\leq j\leq p^{k+1}-1}$. Finally, define the random sequence
\begin{align}
X_n:=\sum_{k=0}^\infty b^k\sum_{\ell=1}^n\sum_{j=1}^{p^{k+1}-1}\bone_{\{J_k=j\}}Y_{k,j}(\ell)=\sum_{k=0}^\infty b^k\sum_{j=1}^{p^{k+1}-1}\bone_{\{J_k=j\}}\sum_{m=1}^{jn}Y_k^j(m)\label{xn}
\end{align}
for $n\in\bN_0$, which converges almost surely since $0<b<1$. Note that if $\bu$ is bounded, then $X_n$ is bounded uniformly in $n$.
\end{example}

\begin{proposition}\label{p:ex2}
The process $\{X_n\}_{n\in\bN_0}$ given in (\ref{xn}) is dt-ss-si with scaling function $b(n)=(|n|_p)^H$, where $H=-\log_p(b)$.
\end{proposition}

\begin{proof}
Since the mixture of dt-ss-si processes with a common scaling function is again a dt-ss-si process, it suffices to prove the result for the case where $\bu$ is deterministic.

The stationarity of the increments of $\{X_n\}_{n\in\bN_0}$ follows directly from the stationarity of $\{Y_k^j(n)\}_{n\in\bN_0}$ hence also of $\{Y_{k,j}(n)\}_{n\in\bN_0}$, and the independence of the sequences with different values of $k$ and $j$.

In order to show the self-similarity, first note that for $q\in\mathcal P\setminus \{p\}$, $k\in\bN_0$, $n\in\bN$ and $1\leq j\leq p^{k+1}-1$,
$$
\sum_{\ell=q(n-1)+1}^{qn}Y_{k,j}(\ell)=\sum_{m=jq(n-1)+1}^{jqn}Y_k^j(m)=Y_{k,[qj]}(n),
$$
where $[qj]$ is the residue of $qj$ modulo $p^{k+1}$. Since $\{j\in\bN: 1\leq j \leq p^{k+1}-1\}=\{[qj]\in\bN: 1\leq j\leq p^{k+1}-1\}$, we have
\begin{align*}
& \left\{\sum_{j=1}^{p^{k+1}-1}\sum_{\ell=1}^{qn}\bone_{\{J_k=j\}}Y_{k,j}(\ell)\right\}_{n\in\bN_0}\\
= & \left\{\sum_{j=1}^{p^{k+1}-1}\bone_{\{J_k=j\}}\sum_{i=1}^n\sum_{\ell=q(i-1)+1}^{qi}Y_{k,j}(\ell)\right\}_{n\in\bN_0}\\
= & \left\{\sum_{j=1}^{p^{k+1}-1}\bone_{\{J_k=j\}}\sum_{i=1}^n Y_{k,[qj]}(i)\right\}_{n\in\bN_0}\\
\dd & \left\{\sum_{[qj]=1}^{p^{k+1}-1}\sum_{i=1}^n \bone_{\{J_k=[qj]\}}Y_{k,[qj]}(i)\right\}_{n\in\bN_0},
\end{align*}
where the last equality in distribution follows from the fact that $J_k$ is uniformly distributed and is independent of everything else. Since the components with different values of $k$ are independent, we must have
$$
\{X_{qn}\}_{n\in\bN_0}\dd \{X_n\}_{n\in\bN_0}.
$$

For $\{X_{pn}\}_{n\in\bN_0}$, note that by the construction of $V_k$, for any $i\in\bN_0$,
$$
\sum_{m=i+1}^{i+p}V_k(m)=V_{k-1}\left(\left\lfloor\frac{i}{p}\right\rfloor +1\right),
$$
where ``$\lfloor\cdot\rfloor$'' gives the largest integer which is smaller than or equal to the variable. Hence
\begin{align*}
\sum_{\ell=p(n-1)+1}^{pn}Y_{k,j}(\ell)&=\sum_{m=jp(n-1)+1}^{jpn}Y_k^j(m)=\sum_{m=jp(n-1)+1}^{jpn}V_k(m+U_k^j)\\
&=\sum_{i=0}^{j-1}V_{k-1}\left(\left\lfloor\frac{p(j(n-1)+i)+U_k^j}{p}
\right\rfloor+1\right)\\
&=\sum_{i=0}^{j-1}V_{k-1}\left(\left\lfloor\frac{U_k^j}{p}
\right\rfloor+ j(n-1)+i+1\right)\\
&=\sum_{m=[j](n-1)+1}^{[j]n}V_{k-1}\left(m+\left\lfloor\frac{U_k^j}{p}\right\rfloor\right),
\end{align*}
where $[j]$ is the residue of $j$ modulo $p^k$, and the last equality follows from the periodicity of $V_{k-1}$.

On the other hand,
$$
Y_{k-1,[j]}(n)=\sum_{m=[j](n-1)+1}^{[j]n}Y_{k-1}^{[j]}(m)=\sum_{m=[j](n-1)+1}^{[j]n}V_{k-1}(m+U^{[j]}_{k-1}).
$$

Since $U_k^j$ is uniformly distributed on $\{0,1,\dots,p^{k+1}-1\}$, $\left\lfloor\frac{U_k^j}{p}\right\rfloor$ is uniformly distributed on $\{0,1,\dots,p^{k}-1\}$. Thus, we have
$$
\left\{\sum_{\ell=p(n-1)+1}^{pn}Y_{k,j}(\ell)\right\}_{n\in\bN}\dd \{Y_{k-1,[j]}(n)\}_{n\in\bN}.
$$

Moreover, because $J_k$ is uniformly distributed on $\{0,\dots, p^{k+1}-1\}$, $[J_k]$ is uniformly distributed on $\{0,\dots, p^k-1\}$, where $[J_k]$ is the residue of $J_k$ modulo $p^k$. Hence by the independence of $U_k^j$ with different values of $k$ and $j$,
$$
\left\{\sum_{j=1}^{p^{k+1}-1}\bone_{\{J_k=j\}}\sum_{m=1}^{pn}Y_{k,j}(m)\right\}_{n\in\bN}\dd \left\{\sum_{j=1}^{p^{k}-1}\bone_{\{J_{k-1}=j\}}\sum_{\ell=1}^nY_{k-1,j}(\ell)\right\}_{n\in\bN}.
$$
Again by independence, a change of index $k'=k-1$ leads to
\begin{align*}
\{X_{pn}\}_{n\in\bN_0}&=\sum_{k=0}^\infty b^k\sum_{j=1}^{p^{k+1}-1}\bone_{\{J_k=j\}}\sum_{m=1}^{pn}Y_{k,j}(m)\\
&=b\sum_{k'=0}^\infty b^{k'}\sum_{j=1}^{p^{k'+1}-1}\bone_{\{J_{k'}=j\}}\sum_{\ell=1}^{n}Y_{k',j}(\ell)\\
&=b\{X_n\}_{n\in\bN_0},
\end{align*}
where the term with $k=0$ on the right hand side of the first line can be dropped since $Y_{0,j}$ has period $p$ and the entries in one period have sum 0.

Therefore, $\{X_n\}_{n\in\bN_0}$ is dt-ss-si with scaling function given by $b(p)=b$ and $b(q)=1$ for all $q\in\mathcal P, q\neq p$.
\end{proof}

\begin{remark}
In the case where $\bu$ is deterministic and has finite support, one can show that the distribution of $X_n$ is also a generalized Bernoulli convolution. That is, when denoting
\[
X_n^{(k)}:=\sum_{j=1}^{p^{k+1}-1}\bone_{\{J_k=j\}}\sum_{m=1}^{jn}Y_k^j(m),
\]
we have that for fixed $n$, $X_n^{(k)}, k\in\bN_0$ are independent and identically distributed. One can also prove that the class of marginal distributions given here belongs to the class given in Example \ref{e:iid}, by making $Y_0^0$ follow the same distribution as $\sum_{k=0}^{J_0} u_{k}$. However, the joint distributions will differ when $p>2$ unless in certain trivial cases, which is not hard to see from the dependence structures of $\{X_n^{(k)}\}_{1\leq n\leq p-1}$. The proof is purely combinatorial and omitted here.
\end{remark}

\begin{remark}
In Example \ref{e:periodic} the processes $\{Y_k^j(n)\}_{n\in\bN_0}$ with different values of $k$ and $j$ share a common $\bu$. Following the same derivation as in the proof of Proposition \ref{p:ex2}, one can easily see that the result will still hold if $\bu$ is replaced by a sequence of independent copies of it, $\{\bu^k\}_{k\in\bN_0}$, as long as the summation in (\ref{xn}) converges. For such processes, $\{Y_k^j(n)\}_{n\in\bN_0}$ with different values of $k$ are independent, while in Example \ref{e:periodic} they are conditionally independent given $\bu$.
\end{remark}

\section{Spectral representation}
Let $\{X_n\}_{n\in\bN_0}$ be a dt-ss-si process of type II, with scaling function $b(n)=(|n|_p)^H$ for $H>0$. Intuitively, since $b(p^i)=(b(p))^i\to 0$ as $i\to\infty$, the distribution of $X_p, X_{p^2}, \dots$ will be more and more concentrated around 0. By the stationarity of the increments, this implies that $X_{n+p^i}-X_n$ is small when $i$ is large. Such an observation leads to the following spectral representation result.

Here and later, we use the notation $e(x)=e^{i2\pi x}$.

\begin{theorem}\label{t:spectral}
Let $p\in\mathcal P$, $\{X_n\}_{n\in\bN_0}$ be a stochastic process satisfying $\bE(|X_1|^2)<\infty$. Then $\{X_n\}_{n\in\bN_0}$ is dt-ss-si of type II with the scaling function $b(n)=(|n|_p)^H, H>0$ if and only if
\[
X_n=\sum_{m=1}^\infty\sum_{0<\ell<p^m,\ p\nmid \ell}A^{(m)}_\ell\left(e\left(\frac{n\ell}{p^m}\right)-1\right), \quad n\in\bN_0
\]
in the sense of convergence in $L^2(\Omega, \mathcal F, \mathbb P)$, where $\{A^{(m)}_\ell\}_{m\in\bN,0<\ell<p^m,p\nmid \ell}$ is an orthogonal sequence in $L^2(\Omega, \mathcal F, \mathbb P)$ and satisfies:
\begin{enumerate}
\item
$$
\{A^{(m)}_\ell\}_{m\in\bN,0<\ell<p^m,p\nmid \ell}\dd\Bigg\{e\left(\frac{\ell}{p^m}\right)A^{(m)}_\ell\Bigg\}_{m\in\bN,0<\ell<p^m,p\nmid \ell};
$$
\item for $q\in\cP$, $q\neq p$,
$$
\{A_\ell^{(m)}\}_{m\in\bN,0<\ell<p^m,p\nmid \ell}\dd \{A_{[q\ell]}^{(m)}\}_{m\in\bN,0<\ell<p^m,p\nmid \ell},
$$
where $[q\ell]$ is the residue of $q\ell$ modulo $p^m$;
\item
$$
\{p^{-H}A_\ell^{(m)}\}_{m\in\bN,0<\ell<p^m,p\nmid \ell}\dd \left\{\sum_{t=0}^{p-1}A_{tp^m+\ell}^{(m+1)}\right\}_{m\in\bN,0<\ell<p^m,p\nmid \ell}.
$$
\end{enumerate}
\end{theorem}

Many results are needed for the proof of Theorem \ref{t:spectral}. We start by introducing the notion of almost periodic functions with values in Banach spaces, which can be found, for example, in \cite{cor}.

\begin{definition}
Let $(X,\Vert\cdot\Vert)$ be a Banach space. A sequence $f:\bZ\to X$ is {\it almost periodic} if for all $\ee>0$, there exists $N(\ee)>0$, such that any consecutive $N(\ee)$ integers contain an integer $T$ with
\[
\Vert f(n+T)-f(n)\Vert<\ee, \quad\text{ for all }n\in\bZ.\]
\end{definition}

Let $\{X_n\}_{n\in\bN_0}$ be a dt-ss-si process of type II. By the stationarity of the increments, $\{Y_n:=X_{n+1}-X_n\}_{n\in\bN_0}$ is a stationary process. Kolmogorov's extension theorem allows us to extend this sequence to $\mathbb Z$ while keeping the stationarity. That is, there exists a stationary process $\{Y'_n\}_{n\in \mathbb Z}$, such that $\{Y'_n\}_{n\in\bN_0}\dd \{Y_n\}_{n\in\bN_0}$. Define
$$
X'_n=\begin{cases}
\sum_{i=0}^{n-1}Y'_i & n\geq0,\\
-\sum_{i=-n}^{-1}Y'_i & n<0,
\end{cases}
$$
then $\{X'_n\}_{n\in \mathbb Z}$ is clearly a dt-ss-si process on $\mathbb Z$, in the sense that it is of stationary increments, and for any $n\in\bN$, there exists $b(n)>0$, such that
$$
\{X'_{nm}\}_{m\in\mathbb Z}\dd \{b(n)X'_m\}_{m\in \mathbb Z}.
$$
Since $\{X'_n\}_{n\in \bN_0}\dd\{X_n\}_{n\in\bN_0}$, $\{X'_n\}_{n\in\mathbb Z}$ is an extension of $\{X_n\}_{n\in\bN_0}$ on $\mathbb Z$. Moreover, by the stationarity of the increments, $\{X'_n\}_{n\in\mathbb Z}$ is an almost periodic sequence in $L^2(\Omega, \mathcal F, \mathbb P)$ if $\{X_n\}_{n\in\bN_0}$ is in $L^2(\Omega, \mathcal F, \mathbb P)$.

\begin{proposition}\label{p:extension}
Let $\{X_n\}_{n\in\bN_0}$ be a dt-ss-si process of type II satisfying $\bE(X_1^2)<\infty$. Then it has an extension on $\mathbb Z$, denoted by $\{X'_n\}_{n\in\mathbb Z}$, which is an almost periodic sequence in $L^2(\Omega, \mathcal F, \mathbb P)$.
\end{proposition}

\begin{proof}
Let $\{X'_n\}_{n\in\mathbb Z}$ be the extension of $\{X_n\}_{n\in\bN_0}$ on $\mathbb Z$ given in the paragraph above Proposition \ref{p:extension}. For any $\ee>0$, take
$$
N(\ee)=p^{\left\lceil -\frac{1}{2H}\log_p\left(\frac{\ee}{\bE(X_1^2)}\right)\right\rceil},
$$
where $\lceil\cdot\rceil$ is the smallest integer which is larger than or equal to the argument. Then, every consecutive $N(\ee)$ integers include a number $\tau$ satisfying $N(\ee)|\tau$. We now have
$$
\sup_{n\in\bN}\bE(|X'_{n+\tau}-X'_n|^2)=\bE((X_\tau)^2)\leq p^{-2H\left\lceil -\frac{1}{2H}\log_p\left(\frac{\ee}{\bE(X_1^2)}\right)\right\rceil}\bE(X_1^2)\leq\ee.
$$
\end{proof}

We call a stochastic process in $L^2(\Omega, \mathcal F, \mathbb P)$ with index set $\bN_0$ an almost periodic process, if it has an extension on $\mathbb Z$ which is almost periodic in $L^2(\Omega, \mathcal F, \mathbb P)$.

By \cite{cor} (Sections 6.3, 1.3), we can associate an almost periodic sequence in $L^2(\Omega, \mathcal F, \mathbb P)$, hence also $\{X_n\}_{n\in\bN_0}$, with a Fourier series:
\begin{equation}\label{e:fouriergeneral}
X_n\sim\sum_{k=1}^\infty A_ke(n\lambda_k), \quad n\in\bN_0
\end{equation}
for some countable set of real numbers $\{\lambda_k\}_{k=1}^\infty$. $\{A_k\}_{k\in\bN}\subset L^2(\Omega, \mathcal F, \mathbb P)$ is given by
\begin{equation}\label{ak}
A_k=\lim_{N\to\infty}\frac{1}{N}\sum_{n=1}^NX_ne(-n\lambda_k), \quad k\in\bN
\end{equation}
in $L^2(\Omega, \mathcal F, \mathbb P)$.
If moreover, the right hand side of (\ref{e:fouriergeneral}) is uniformly convergent in $L^2(\Omega, \mathcal F, \mathbb P)$, then
\begin{equation*}\label{e:fouriereq}
X_n=\sum_{k=1}^\infty A_ke(n\lambda_k), \quad n\in\bN_0,
\end{equation*}
where the infinite sum is in the sense of $L^2(\Omega, \mathcal F, \mathbb P)$. We do not have the convergence at this moment, but will establish it using the properties of the process $\{X_n\}_{n\in\bN_0}$.

The following lemma shows that the coefficient $A_k$ can be nonzero only if the corresponding $\lambda_k$ is a p-adic rational.

\begin{lemma}\label{l:Akdescription}
Let $\{X_n\}_{n\in\bN_0}$ be a dt-ss-si process of type II satisfying $\bE(X_1^2)<\infty$, then
$$
X_n\sim\sum_{k=1}^\infty A_ke(n\lambda_k), \quad n\in\bN_0,
$$
where $\{A_k\}_{k\in\bN}\subset L^2(\Omega, \mathcal F, \mathbb P)$ and $\{\lambda_k\}_{k\in\bN}$ is the set of $p$-adic rationals in $[0,1)$.
\end{lemma}

\begin{proof}
 It suffices to show $A_k=0$ in (\ref{e:fouriergeneral}) for $\lambda_k$ not of the form $\ell p^{-m}$ where $\ell\in\bN_0, m\in\bN$. Let $\lambda\in\bR$ be such that $p^m\lambda$ is not an integer for any $m\in\bN$. Using (\ref{ak}), for every $m\in\bN$, the coefficient corresponding to $\lambda$, denoted by $a(\lambda)$, satisfies
$$
\bE\left(|a(\lambda)|^2\right)=\lim_{N\to\infty}\bE\left(\left|\frac{1}{Np^m}\sum_{n=1}^{Np^m}X_ne(-n\lambda)\right|^2\right).
$$

By Cauchy-Schwarz inequality,
\begin{align*}
&\bE\left(\left|\sum_{n=1}^{Np^m}X_ne(-n\lambda)\right|^2\right)\\
=&\bE\left(\left|\sum_{j=1}^{p^m}\sum_{k=0}^{N-1}e(-(kp^m+j)\lambda)X_j+\sum_{j=1}^{p^m}\sum_{k=1}^{N-1}e(-(kp^m+j)\lambda)(X_{kp^m+j}-X_j)\right|^2\right)\\
\leq& Np^m\sum_{j=1}^{p^m}\bE\left(\left|\sum_{k=0}^{N-1}e(-(kp^m+j)\lambda)X_j\right|^2\right)\\
&+Np^m\sum_{j=1}^{p^m}\sum_{k=1}^{N-1}\bE\left(\left|e(-(kp^m+j)\lambda)(X_{kp^m+j}-X_j)\right|^2\right)\\
=&Np^m\left(\sum_{j=1}^{p^m}\left|\sum_{k=0}^{N-1}e(-(kp^m+j)\lambda)\right|^2\bE(X_j^2)+\sum_{j=1}^{p^m}\sum_{k=1}^{N-1}\bE((X_{kp^m+j}-X_j)^2)\right)\\
\leq& Np^{2m}\left|\sum_{k=0}^{N-1}e(-(kp^m+j)\lambda)\right|^2\max_{1\leq j\leq p^m}\bE(X_j^2)+Np^m\sum_{j=1}^{p^m}\sum_{k=1}^{N-1}p^{-2mH}\bE(X_1^2)\\
\leq & Np^{2m}\left|\sum_{k=0}^{N-1}e(-(kp^m+j)\lambda)\right|^2\max_{1\leq j\leq p^m}\bE(X_j^2)+N^2p^{2m}p^{-2mH}\bE(X_1^2).
\end{align*}

Hence
$$
\bE\left(|a(\lambda)|^2\right)\leq \bE(X _1^2)\lim_{N\to\infty}\left|\frac{1}{\sqrt{N}}\sum_{k=0}^{N-1}e(-(kp^m+j)\lambda)\right|^2+\bE(X_1^2)p^{-2mH}.
$$

As $p^m\lambda$ is not an integer, it is easy to see that
\begin{align*}
\left|\frac{1}{\sqrt{N}}\sum_{k=0}^{N-1}e(-(kp^m+j)\lambda)\right|&=\left|\frac{1}{\sqrt{N}}\frac{e(-j\lambda)-e(-(Np^m+j)\lambda)}{1-e(-p^m\lambda)}\right|\\
&\leq\frac{2}{\sqrt{N}}\left|\frac{1}{1-e(-p^m\lambda)}\right|,
\end{align*}
which converges to $0$ as $N\to\infty$. Therefore
$$
\bE\left(|a(\lambda)|^2\right)\leq p^{-2mH}\bE(X_1^2).
$$

Since this holds for all $m\in\bN$, letting $m\to\infty$ leads to the conclusion that $A_k$ can only be non-zero if the corresponding $\lambda_k$ is a $p$-adic rational. Finally, since $e(x)$ has period $1$, $\{e(n\lambda)\}_{n\in\bN_0}=\{e(n(\lambda+1))\}_{n\in\bN_0}$. Hence we only need $p$-adic rationals in $[0,1)$.\end{proof}

\begin{remark}
The above lemma also holds in $L^1(\Omega)$ if $\bE(|X_1|)<\infty$. The proof is essentially the same by replacing the Cauchy-Schwarz inequality by the triangle inequality. For simplicity, we only consider the $L^2$ case. Also note that for $L^1\setminus L^2$, the convergence of the associated Fourier series is not guaranteed, hence although still valid, the result of Lemma \ref{l:Akdescription} becomes less important.
\end{remark}

Lemma \ref{l:Akdescription} allows us to further explore the detailed impact of the stationarity of the increments and the self-similarity of the process to the representation (\ref{e:fouriergeneral}). We start from the following simple observation about the increment process.

\begin{lemma}\label{l:incrementAk}
Let $\{X_n\}_{n\in\bN_0}$ be a dt-ss-si process of type II satisfying $\bE(X_1^2)<\infty$ and
$$
X_n\sim\sum_{k=1}^\infty A_ke(n\lambda_k), \quad n\in\bN_0.
$$
Then its increment process $\{\tilde{X}_n\}_{n\in\bN_0}$, given by
\[
\tilde{X}_n=X_{n+1}-X_n, \quad n\in \bN_0,
\]
is almost periodic in $L^2(\Omega, \mathcal F, \mathbb P)$ and stationary. Moreover,
\[
\tilde{X}_n\sim\sum_{k=1}^\infty \tilde A_ke(n\lambda_k), \quad n\in \bN_0,
\]
where $\tilde A_k=A_k(e(\lambda_k)-1)$.
\end{lemma}

\begin{proof}
The stationarity is trivial, and the almost periodicity follows directly from
\begin{multline*}
\bE(|(X_{p^m+n+1}-X_{p^m+n})-(X_{n+1}-X_n)|^2)\\
\leq 2(\bE(|X_{p^m+n+1}-X_{n+1}|^2)+\bE(|X_{p^m+n}-X_n|^2)).
\end{multline*}
The representation is obvious from (\ref{ak}) and the relation $\tilde{X}_n=X_{n+1}-X_n$.
\end{proof}

As a consequence of Lemmas \ref{l:Akdescription} and \ref{l:incrementAk}, the increment process $\{\tilde X_n\}_{n\in\bN_0}$ is associated with the Fourier series
$$
\sum_{m\in\bN}\sum_{0<\ell<p^m,\ p\nmid \ell} \tilde A^{(m)}_\ell e\left(\frac{n\ell}{p^m}\right).
$$
Intuitively, the original single summation in Lemma \ref{l:incrementAk} can be divided into different layers according to the $p$-adic norm of $\lambda_k$. Based on this decomposition, the stationarity of $\{\tilde X_n\}_{n\in\bN_0}$ implies a rotation-invariant property of the coefficients $\{\tilde A^{(m)}_\ell\}_{m\in\bN, 0<\ell<p^m,p\nmid \ell}$, which further implies the orthogonality.

\begin{lemma}\label{l:stationary}
Let $\{Y_n\}_{n\in\bN_0}$ be an almost periodic process in $L^2(\Omega, \mathcal F, \mathbb P)$ such that
$$
Y_n\sim\sum_{m\in\bN}\sum_{0<\ell<p^m,\ p\nmid \ell} \tilde A^{(m)}_\ell e\left(\frac{n\ell}{p^m}\right), \quad n\in\bN_0.\label{ynrep}
$$
If $\{Y_n\}_{n\in\bN_0}$ is stationary, then
\begin{equation}\label{e:rotation}
\left\{\tilde A^{(m)}_\ell\right\}_{m\in\bN,0<\ell<p^m,p\nmid \ell}\dd\left\{e\left(\frac{\ell}{p^m}\right)\tilde A^{(m)}_\ell\right\}_{m\in\bN,0<\ell<p^m,p\nmid \ell},
\end{equation}
in particular, $\{\tilde A^{(m)}_\ell\}_{m\in\bN,0<\ell<p^m,p\nmid \ell}$ is an orthogonal sequence in $L^2(\Omega, \mathcal F, \mathbb P)$.
\end{lemma}

\begin{proof}
%Assume (\ref{e:rotation}) holds, then
%\[
%\{Y_n\}_{n\in\bN_0}\dd   \left\{\sum_{m\in\bN}\sum_{0<l<p^m,\ p\nmid l} \tilde A^{(m)}_le\left({\frac{l}{p^m}}\right)e\left(\frac{nl}{p^m}\right)\right\}_{n\in\bN_0}  \dd\{Y_{n+1}\}_{n\in\bN_0},
%\]
%Hence $\{Y_n\}_{n\in\bN_0}$ is stationary.

Assume $\{Y_n\}_{n\in\bN_0}$ is stationary. Since the process $\{Y_{n+1}\}_{n\in\bN_0}$ is also almost periodic and in $L^2(\Omega, \mathcal F, \mathbb P)$, it is associated with a Fourier series as well. The coefficient $A_k'$ corresponding to $\lambda_k$ is given by
\begin{align*}
A_k'&=\lim_{N\to\infty}\frac{1}{N}\sum_{n=1}^N Y_{n+1}e(-n\lambda_k)\\
&=\lim_{N\to\infty}\frac{1}{N}\sum_{n=2}^{N+1} Y_{n}e(-(n-1)\lambda_k)\\
&=e(\lambda_k)A_k.
\end{align*}

As $\{Y_n\}_{n\in\bN_0}\dd \{Y_{n+1}\}_{n\in\bN_0}$, by the uniqueness of the associated Fourier series, the coefficients of the corresponding terms must also have the same distribution. Hence (\ref{e:rotation}) holds.

Furthermore, for $i=1,2$, let $m_i\in\bN$, $\ell_i$ be such that $0<\ell_i<p^{m_i}$ and $p\nmid \ell_i$. If $(m_1, \ell_1)\neq (m_2, \ell_2)$, then
$$
\sum_{k=0}^{p^{m_1\vee m_2}-1}e\left(\frac{k\ell_1}{p^{m_1}}-\frac{k\ell_2}{p^{m_2}}\right)=0.
$$
Hence by the rotation-invariance that we just proved,
\begin{align*}
0&=\bE\left(\sum_{k=0}^{p^{m_1\vee m_2}-1}e\left(\frac{k\ell_1}{p^{m_1}}\right)\tilde A^{(m_1)}_{\ell_1}\overline{e\left(\frac{k\ell_2}{p^{m_2}}\right)\tilde A^{(m_2)}_{\ell_2}}\right)\\
&=p^{m_1\vee m_2}\bE\left(\tilde A_{\ell_1}^{(m_1)}\overline{\tilde A^{(m_2)}_{\ell_2}}\right).
\end{align*}
Thus, $\tilde A^{(m_1)}_{\ell_1}$ and $\tilde A^{(m_2)}_{\ell_2}$ are orthogonal.
\end{proof}

Lemma \ref{l:Akdescription} also allows us to directly rewrite the representation (\ref{e:fouriergeneral}) as
$$
X_n\sim A_1+\sum_{m=1}^\infty\sum_{0<\ell<p^m, p\nmid \ell}A^{(m)}_\ell e\left(\frac{n\ell}{p^m}\right), \quad n\in\bN_0,
$$
where $A_1$ is the coefficient corresponding to $\lambda_1=0$, \textit{i.e.}, the constant term. As a result, Lemma \ref{l:stationary} has the following simple corollary for processes with stationary increments.

\begin{corollary}\label{c:Akstationary}
Let $\{X_n\}_{n\in\bN_0}$ be an almost periodic process in $L^2(\Omega, \mathcal F, \mathbb P)$ with the representation
\[
X_n\sim A_1+\sum_{m=1}^\infty\sum_{0<\ell<p^m, p\nmid \ell}A^{(m)}_\ell e\left(\frac{n\ell}{p^m}\right), \quad n\in\bN_0.
\]
If $\{X_n\}_{n\in\bN_0}$ has stationary increments, then
\begin{equation}\label{e:Akrotation}
\left\{A^{(m)}_\ell\right\}_{m\in\bN,0<\ell<p^m,p\nmid \ell}\dd\left\{e\left(\frac{\ell}{p^m}\right)A^{(m)}_\ell\right\}_{m\in\bN,0<\ell<p^m,p\nmid \ell},
\end{equation}
in particular, $\left\{A^{(m)}_\ell\right\}_{m\in\bN,0<\ell<p^m,p\nmid \ell}$ is an orthogonal sequence in $L^2(\Omega, \mathcal F, \mathbb P)$.
\end{corollary}

The proof of this corollary is trivial by noticing that $A^{(m)}_\ell$ and $\tilde A^{(m)}_{\ell}$ are different only by a deterministic multiplicative factor.

We have seen how the stationarity of the increments has an impact on the coefficients for the increment process and therefore, also on the coefficients for the original process. Next, we discuss an impact of the self-similarity to the coefficients in the representation.

\begin{lemma}\label{l:Akselfsimilar}
Let $\{X_n\}_{n\in\bN_0}$ be an almost periodic process with $\bE(X_1^2)<\infty$ and the representation
\[
X_n\sim A_1+\sum_{m=1}^\infty\sum_{0<\ell<p^m, p\nmid \ell}A^{(m)}_\ell e\left(\frac{n\ell}{p^m}\right), \quad n\in\bN_0.
\]
If $\{X_n\}_{n\in\bN_0}$ is discrete-time self-similar with scaling function $b(n)=(|n|_p)^H$, $H>0$, then
\begin{equation}
\left\{p^{-H}A_\ell^{(m)}\right\}_{m\in\bN, 0<\ell<p^m, p\nmid \ell}\dd \left\{\sum_{t=0}^{p-1}A_{tp^m+\ell}^{(m+1)}\right\}_{m\in\bN, 0<\ell<p^m, p\nmid \ell}.\label{p}
\end{equation}
\end{lemma}

\begin{proof}
For any $m\in\bN$ and $\ell$ satisfying $0<\ell<p^m, p\nmid \ell$,
\begin{align*}
&~\sum_{t=0}^{p-1}A^{(m+1)}_{tp^{m}+\ell}\\
=&~\lim_{N\to\infty}\frac{1}{Np^{m+1}}\sum_{t=0}^{p-1}\sum_{n=1}^{Np^{m+1}}e\left(-\frac{n(tp^m+\ell)}{p^{m+1}}\right)X_n\\
=&~\lim_{N\to\infty}\frac{1}{Np^{m+1}}\sum_{t=0}^{p-1}\sum_{k=0}^{N-1}\sum_{j=1}^{p^{m+1}}e\left(-\frac{(kp^{m+1}+j)(tp^m+\ell)}{p^{m+1}}\right)X_{kp^{m+1}+j}\\
=&~\lim_{N\to\infty}\frac{1}{Np^{m+1}}\sum_{j=1}^{p^{m+1}}\sum_{t=0}^{p-1}e\left(-\frac{j(tp^m+\ell)}{p^{m+1}}\right)\sum_{k=0}^{N-1}X_{kp^{m+1}+j}.
\end{align*}

Note that the summation
$$
\sum_{t=0}^{p-1}e\left(-\frac{j(tp^m+\ell)}{p^{m+1}}\right)
$$
is non-zero only if $p|j$, in which case it takes value $pe\left(-\frac{j\ell}{p^{m+1}}\right)$. Therefore, by letting $j'=j/p$, we have
$$
\sum_{t=0}^{p-1}A^{(m+1)}_{tp^{m}+\ell}=\lim_{N\to\infty}\frac{1}{Np^{m}}\sum_{j'=1}^{p^{m}}e\left(-\frac{j'\ell}{p^{m}}\right)\sum_{k=0}^{N-1}X_{kp^{m+1}+pj'}.
$$

Recall that
$$
\{X_{pn}\}_{n\in\bN_0}\dd p^{-H}\{X_n\}_{n\in\bN_0},
$$
hence
\begin{align*}
&~\left\{\sum_{t=0}^{p-1}A^{(m+1)}_{tp^{m}+\ell}\right\}_{m\in\bN, 0<\ell<p^m, p\nmid \ell}\\
\dd &~\left\{\lim_{N\to\infty}\frac{p^{-H}}{Np^{m}}\sum_{j'=1}^{p^{m}}e\left(-\frac{j'\ell}{p^{m}}\right)\sum_{k=0}^{N-1}X_{kp^{m}+j'}\right\}_{m\in\bN, 0<\ell<p^m, p\nmid \ell}\\
= &~\left\{p^{-H}\lim_{N\to\infty}\frac{1}{Np^{m}}\sum_{k=0}^{N-1}\sum_{j'=1}^{p^{m}}e\left(-\frac{(kp^m+j')\ell}{p^{m}}\right)X_{kp^{m}+j'}\right\}_{m\in\bN, 0<\ell<p^m, p\nmid \ell}\\
= &~\left\{p^{-H}A_\ell^{(m)}\right\}_{m\in\bN, 0<\ell<p^m, p\nmid \ell}.
\end{align*}
\end{proof}

Corollary \ref{c:Akstationary} and Lemma \ref{l:Akselfsimilar} together guarantee a very important result: the convergence of the Fourier series associated with a dt-ss-si process of type II in $L^2(\Omega, \mathcal F, \mathbb P)$.

\begin{proposition}\label{p:convergence}
Let $\{A^{(m)}_\ell\}_{m\in\bN,0<\ell<p^m,p\nmid \ell}$ be an orthogonal sequence in $L^2(\Omega, \mathcal F, \mathbb P)$ satisfying (\ref{p}). Then the Fourier series
$$
\sum_{m=1}^\infty\sum_{0<\ell<p^m,\ p\nmid \ell}A^{(m)}_\ell e\left(\frac{n\ell}{p^m}\right)
$$
converges uniformly in $L^2(\Omega, \mathcal F, \mathbb P)$.
\end{proposition}

\begin{proof}
Orthogonality implies that
$$
\bE\left(\left|\sum_{m=M}^N\sum_{0<\ell<p^m,\ p\nmid \ell}A^{(m)}_\ell e\left(\frac{n\ell}{p^m}\right)\right|^2\right)=\sum_{m=M}^N\sum_{0<\ell<p^m,\ p\nmid \ell}\bE(|A^{(m)}_\ell|^2).
$$
On the other hand, (\ref{p}), together with the orthogonality, also gives
\begin{align*}
\sum_{0<\ell<p^{m+1},\ p\nmid \ell}\bE(|A^{(m+1)}_\ell|^2)&=\sum_{\ell=1}^{p^m}\bone_{p\nmid \ell}\sum_{t=0}^{p-1}\bE(|A^{(m+1)}_{tp^m+\ell}|^2)\\
&=\sum_{\ell=1}^{p^m}\bone_{p\nmid \ell}\bE(|\sum_{t=0}^{p-1}A^{(m+1)}_{tp^m+\ell}|^2)\\
&=\sum_{\ell=1}^{p^m-1}\bone_{p\nmid \ell}p^{-2H}\bE(|A^{(m)}_\ell|^2).
\end{align*}
Hence by induction,
\[
\sum_{0<\ell<p^{m+1},\ p\nmid \ell}\bE(|A^{(m+1)}_\ell|^2)=p^{-2mH}\sum_{\ell=1}^{p-1}\bE(|A^{(1)}_\ell|^2).
\]
Thus,
\[
\bE\left(\left|\sum_{m=M}^N\sum_{0<\ell<p^m,\ p\nmid \ell}A^{(m)}_\ell e\left(\frac{n\ell}{p^m}\right)\right|^2\right)\leq \frac{p^{-2(M-1)H}-p^{-2NH}}{p^{2H}-1}\sum_{\ell=1}^{p-1}\bE(|A^{(1)}_\ell|^2)
\]
which converges uniformly to 0 as $M, N\to\infty$. Hence the Fourier series converges uniformly in $L^2(\Omega, \mathcal F, \mathbb P)$.
\end{proof}

As a direct consequence of Proposition \ref{p:convergence}, all the Fourier series discussed in this section converge and hence are equal to the original sequences. In other words, the ``$\sim$'' can be now replaced by ``=''. This allows us to easily expand Corollary \ref{c:Akstationary} to a two-directional result.

\begin{proposition}\label{p:stationary}
Let $\{X_n\}_{n\in\bN_0}$ be an almost periodic process in $L^2(\Omega, \mathcal F, \mathbb P)$ with the representation
\[
X_n= A_1+\sum_{m=1}^\infty\sum_{0<\ell<p^m, p\nmid \ell}A^{(m)}_\ell e\left(\frac{n\ell}{p^m}\right), \quad n\in\bN_0.
\]
Then $\{X_n\}_{n\in\bN_0}$ has stationary increments if and only if (\ref{e:Akrotation}) holds.
\end{proposition}

\begin{proof}
The ``only if'' part is exactly Corollary \ref{c:Akstationary}. For the ``if'' part, note that for the increment process $\{\tilde X_n\}_{n\in\bN_0}$, we have
$$
\tilde X_n=\sum_{m=1}^\infty\sum_{0<\ell<p^m,\ p\nmid \ell} \tilde A^{(m)}_\ell e\left(\frac{n\ell}{p^m}\right), \quad n\in\bN_0,
$$
where $\tilde A^{(m)}_\ell=A^{(m)}_\ell(e(\ell/p^m)-1)$. Because of the relation between $\tilde A^{(m)}_\ell$ and $A^{(m)}_\ell$, (\ref{e:Akrotation}) is equivalent to
$$
\left\{\tilde A^{(m)}_\ell\right\}_{m\in\bN,0<\ell<p^m,p\nmid \ell}\dd\left\{e\left(\frac{\ell}{p^m}\right)\tilde A^{(m)}_\ell\right\}_{m\in\bN,0<\ell<p^m,p\nmid \ell}.
$$
With this condition, it is obvious that
$$
\{\tilde X_n\}_{n\in\bN_0}\dd   \left\{\sum_{m=1}^\infty\sum_{0<\ell<p^m,\ p\nmid \ell} \tilde A^{(m)}_\ell e\left({\frac{\ell}{p^m}}\right)e\left(\frac{n\ell}{p^m}\right)\right\}  =\{\tilde X_{n+1}\}_{n\in\bN_0}.
$$
\end{proof}

Let $\{X_n\}_{n\in\bN_0}$ be a dt-ss-si process of type II with representation
$$
X_n= A_1+\sum_{m=1}^\infty\sum_{0<\ell<p^m, p\nmid \ell}A^{(m)}_\ell e\left(\frac{n\ell}{p^m}\right), \quad n\in\bN_0.
$$
Since $X_0=0$ almost surely, we must have
$$
A_1=-\sum_{m=1}^\infty\sum_{0<\ell<p^m, p\nmid \ell}A^{(m)}_\ell.
$$
Thus, the representation can be rewritten as
\begin{equation}\label{e:representation}
X_n=\sum_{m=1}^\infty\sum_{0<\ell<p^m, p\nmid \ell}A^{(m)}_\ell \left(e\left(\frac{n\ell}{p^m}\right)-1\right), \quad n\in\bN_0.
\end{equation}

With Proposition \ref{p:convergence} and (\ref{e:representation}), Lemma \ref{l:Akselfsimilar} also gets a significant extension, which includes a condition corresponding to the rescaling invariance of the distribution of $\{X_n\}_{n\in\bN_0}$ with factor $q\in\mathcal P, q\neq p$, as well as the sufficiency of the conditions.

\begin{proposition}\label{p:selfsimilar}
Let $\{X_n\}_{n\in\bN_0}$ be an almost periodic process in $L^2(\Omega, \mathcal F, \mathbb P)$ with the representation
\begin{equation*}\label{e:sim}
X_n=\sum_{m=1}^\infty\sum_{0<\ell<p^m, p\nmid \ell}A^{(m)}_\ell \left(e\left(\frac{n\ell}{p^m}\right)-1\right), \quad n\in\bN_0.
\end{equation*}
Then $\{X_n\}_{n\in\bN_0}$ is discrete-time self-similar with scaling function $b(n)=(|n|_p)^H$ for $H>0$ if and only if (\ref{p}) holds, and
\begin{equation}\label{q}
\{A^{(m)}_\ell\}_{m\in\bN,0<\ell<p^m,p\nmid \ell}\dd \{A^{(m)}_{[q\ell]}\}_{m\in\bN,0<\ell<p^m,p\nmid \ell},
\end{equation}
where $[q\ell]$ is the residue of $q\ell$ modulo $p^m$.
\end{proposition}

\begin{proof}
Assume $\{X_n\}_{n\in\bN_0}$ is a discrete-time self-similar process with scaling function $b(n)=(|n|_p)^H$ for $H>0$, then (\ref{p}) holds by Lemma \ref{l:Akselfsimilar}. Moreover, note that for $n\in\bN_0$,
\begin{align*}
X_{qn}&=\sum_{m=1}^\infty\sum_{0<\ell<p^m, p\nmid \ell}A^{(m)}_\ell \left(e\left(\frac{nq\ell}{p^m}\right)-1\right)\\
&=\sum_{m=1}^\infty\sum_{0<\ell<p^m, p\nmid \ell}A^{(m)}_\ell \left(e\left(\frac{n[q\ell]}{p^m}\right)-1\right).
\end{align*}

On the other hand, since $\{X_{qn}\}_{n\in\bN_0}\dd \{X_n\}_{n\in\bN_0}$, we have
\begin{align*}
\{X_{qn}\}_{n\in\bN_0}&\dd\left\{\sum_{m=1}^\infty\sum_{0<\ell<p^m, p\nmid \ell}A^{(m)}_\ell \left(e\left(\frac{n\ell}{p^m}\right)-1\right)\right\}_{n\in\bN_0}\\
&=\left\{\sum_{m=1}^\infty\sum_{0<\ell<p^m, p\nmid \ell}A^{(m)}_{[q\ell]} \left(e\left(\frac{n[q\ell]}{p^m}\right)-1\right)\right\}_{n\in\bN_0},
\end{align*}
where the second equality follows from the simple observation $\{\ell: 0<\ell<p^m, p\nmid \ell\}=\{[q\ell]: 0<\ell<p^m, p\nmid \ell\}$. By the uniqueness of the Fourier expansion, we must have
$$
\{A^{(m)}_\ell\}_{m\in\bN,0<\ell<p^m,p\nmid \ell}\dd \{A^{(m)}_{[q\ell]}\}_{m\in\bN,0<\ell<p^m,p\nmid \ell}.
$$

Conversely, assume (\ref{p}) and (\ref{q}) hold. Then for each $q\in\cP\setminus\{p\}$,
\begin{align*}
\{X_{qn}\}_{n\in\bN_0}&=\left\{\sum_{m=1}^\infty\sum_{0<\ell<p^m,\ p\nmid \ell}A^{(m)}_\ell\left(e\left(\frac{nq\ell}{p^m}\right)-1\right)\right\}_{n\in\bN_0}\\
&\dd\left\{\sum_{m=1}^\infty\sum_{0<\ell<p^m,\ p\nmid \ell}A^{(m)}_{[q\ell]}\left(e\left(\frac{nq\ell}{p^m}\right)-1\right)\right\}_{n\in\bN_0}\\
&= \left\{\sum_{m=1}^\infty\sum_{0<[q\ell]<p^m,\ p\nmid \ell}A^{(m)}_{[q\ell]}\left(e\left(\frac{n[q\ell]}{p^m}\right)-1\right)\right\}_{n\in\bN_0}\\
&=\{X_n\}_{n\in\bN_0}.
\end{align*}

Similarly,
\begin{align*}
\{X_{pn}\}_{n\in\bN_0}&=\left\{\sum_{m=2}^\infty\sum_{0<\ell<p^m,\ p\nmid \ell}A^{(m)}_\ell\left(e\left(\frac{n\ell}{p^{m-1}}\right)-1\right)\right\}_{n\in\bN_0}\\
&=\left\{\sum_{m=2}^\infty\sum_{0<\ell<p^{m-1},\ p\nmid \ell}\left(e\left(\frac{n\ell}{p^{m-1}}\right)-1\right)\sum_{t=0}^{p-1}A^{(m)}_{tp^{m-1}+\ell}\right\}_{n\in\bN_0}\\
&=\left\{\sum_{m=1}^\infty\sum_{0<\ell<p^{m},\ p\nmid \ell}\left(e\left(\frac{n\ell}{p^{m}}\right)-1\right)\sum_{t=0}^{p-1}A^{(m+1)}_{tp^{m}+\ell}\right\}_{n\in\bN_0}\\
&\dd\left\{\sum_{m=1}^\infty\sum_{0<\ell<p^{m},\ p\nmid \ell}\left(e\left(\frac{n\ell}{p^{m}}\right)-1\right)p^{-H}A^{(m)}_\ell\right\}_{n\in\bN_0}\\
&=\{p^{-H}X_n\}_{n\in\bN_0}.
\end{align*}

Thus, $\{X_n\}_{n\in\bN_0}$ is self-similar with scaling function $b(n)=(|n|_p)^H$.
\end{proof}

Combining the results of Lemma \ref{l:Akdescription}, Propositions \ref{p:convergence}, \ref{p:stationary} and \ref{p:selfsimilar} immediately leads to Theorem \ref{t:spectral}.

\section*{Acknowledgement} The authors would like to thank Gennady Samorodnitsky, Wanchun Shen, Ruodu Wang and Yimin Xiao for their valuable inputs. Yi Shen acknowledges financial support from the Natural Sciences and Engineering Research Council of Canada (RGPIN-2014-04840).


\begin{thebibliography}{11}

\vspace{9pt}

\bibitem{cor} Corduneanu, C. (1989). {\it Almost periodic functions}. Chelsea Pub Co.

\bibitem{emb} Embrechts, P. and Maejima, M. (2002). {\it Selfsimilar processes}. Princeton University Press.

\bibitem{gef} Gefferth, A., Veitch, D., Maricza, I., Moln\'{a}r, S. and Ruzsa, I. (2003). The nature of discrete second-order self-similarity. {\it Advances in Applied Probability}, 35(2), 395-416.

\bibitem{obr} O'Brien, G. L. and Vervaat, W. (1983). Marginal distributions of self-similar processes with stationary increments. {\it Zeitschrift f\"{u}r Wahrscheinlichkeitstheorie und verwandte Gebiete}, 64(1), 129-138.

\bibitem{per} Peres, Y., Schlag, W. and Solomyak, B. (2000). Sixty years of Bernoulli convolutions. In {\it Fractal geometry and stochastics II} (pp. 39-65). Birkh\"{a}user, Basel.
\bibitem{puc} Puccetti, G., Rigo, P., Wang, B. and Wang, R. (2018). Centers of probability measures without the mean. {\it Journal of Theoretical Probability}, https://doi.org/10.1007/s10959-018-0815-3.
\bibitem{sam} Samorodnitsky, G. and Taqqu, M. S. (1994). {\it Stable non-Gaussian random processes: stochastic models with infinite variance}. CRC Press.
\bibitem{sam16} Samorodnitsky, G. (2016). {\it Stochastic processes and long range dependence}. Springer.
\bibitem{ver} Vervaat, W. (1985). Sample path properties of self-similar processes with stationary increments. {\it The Annals of Probability}, 13(1), 1-27.

\end{thebibliography}
\end{document}